\documentclass[11pt]{article}
\usepackage{setspace}
\usepackage{amsmath,amssymb}
\usepackage{agor}
\usepackage{color}
\usepackage[square,numbers,sort]{natbib}

\newcommand{\ignore}[1]{}

\usepackage{geometry,enumerate,todonotes}
\usepackage{graphicx,wrapfig,soul}
\usepackage{bbm}
\usepackage{subfig}

\title{A note on quadratic constraints with indicator variables: Convex hull description and perspective relaxation}
\author{Andr\'es G\'omez\thanks{Department of Industrial and System Engineering, University of Southern California, Los Angeles, CA 90089. Phone: 510-499-2418. Email: \texttt{gomezand@usc.edu}.}  \and Weijun Xie\thanks{H. Milton Stewart School of Industrial and Systems Engineering, Georgia Institute of Technology, Atlanta, GA 30332 Email: \texttt{wxie@gatech.edu}.} }

\begin{document}
	\maketitle
	
	\begin{abstract}
	    In this paper, we study the mixed-integer nonlinear set given by a separable quadratic constraint on continuous variables, where each continuous variable is controlled by an additional indicator. This set occurs pervasively in optimization problems with uncertainty and in machine learning. We show that optimization over this set is NP-hard. Despite this negative result, we characterize the structure of the convex hull, and show that it can be formally studied using polyhedral theory. Moreover, we show that although perspective relaxation in the literature for this set fails to match the structure of its convex hull, it is guaranteed to be a close approximation.
	\end{abstract}

\section{Introduction}
In this paper, given $Z\subseteq \{0,1\}^n$, we study set 
	$$X\defeq \{(\bm x,\bm z)\in \R^n\times Z: \|\bm x\|_2^2\leq 1,\;  \bm x\circ(\ones-\bm z)=0 \},$$
	where $\ones$ is a vector of $1$s and ``$\circ$" denotes the Hadamard (entry-wise) product of vectors. 
Set $X$ is non-convex due to the binary constraints encoded by $Z$, as well as the complementarity constraints $\bm x\circ(\ones-\bm z)=0$ linking the continuous and binary variables. Observe that arbitrary separable quadratic constraints of the form
$\sum_{i=1}^n (d_ix_i)^2\leq b$ can be modeled with $X$ as well through the change of variables $\bar x_i\defeq (d_i/\sqrt{b})x_i$. Note that since any $(\bm{x},\bm{z})\in X$ satisfies $|x_i|\leq 1$, the complementarity constraints can be linearized as the big-M constraints 
\begin{equation}\label{eq:bigM}
|x_i|\leq z_i,\quad i=1,\dots,n.
\end{equation}
Our overall goal is to understand and characterize the convex hull of $X$, denoted as $\conv(X)$. Throughout the paper, for simplicity, we use the following convention for division by $0$: $a/b=0$ if $a=b=0$, and $a/b=\infty$ ($-\infty$) if $b=0$ and $a>0$ (or $a<0$).

\subsection{Applications} \label{sec:apps}
Set $X$ arises pervasively in practice. We now discuss three settings where it plays a key role.

\paragraph{Sparse PCA} Set $X$ arises directly in sparse principal component analysis problems \cite{d2008optimal,zou2006sparse,dey2021using,li2020exact}, a fundamental problem in statistics which can be formulated as 
\begin{subequations}\label{eq:sparsePCA}
\begin{align}
\max\;& \bm{x'\Sigma x}\label{eq:sparsePCA_obj}\\
\text{s.t.}\;&\|\bm{x}\|_2^2\leq 1,\; \|\bm{z}\|_1\leq k,\; \bm x\circ(\ones-\bm z)=0,\label{eq:sparsePCA_constr}
\end{align}
\end{subequations}
where $\bm{\Sigma}\succeq 0$ and $k\in \Z_+$ is a parameter controlling the sparsity of the solution. Observe that the feasible region given by constraints \eqref{eq:sparsePCA_constr} corresponds exactly to $X$ with set $Z=\{\bm{z}\in \{0,1\}^n: \|\bm{z}\|_1\leq k\}$. Thus, understanding $\conv(X)$ is critical to designing better convex approximations of~\eqref{eq:sparsePCA}.

\paragraph{General convex quadratic constraints} Given $\bm{\Sigma}\succeq 0$, consider the system of inequalities
\begin{equation}\label{eq:quadConstr}\bm{y'\Sigma y}\leq b,\; \bm{y}\circ(\ones-\bm z)=0,\; \bm y\in \R^n,\;\bm{z}\in Z\subseteq \{0,1\}^n.\end{equation}
System \eqref{eq:quadConstr} arises for example in mean-variance optimization problems \cite{atamturk2008b}, where the quadratic constraint is used to impose an upper bound on the risk (variance) of the solution. While system \eqref{eq:quadConstr} involves a non-separable quadratic constraint, a study of set $\conv(X)$ can be still used to construct strong convex relaxations. Indeed, if $\bm{\Sigma}=\bm{D}+\bm{R}$ where $\bm{R}\succeq 0$, $\bm{D}\succ 0$ and diagonal, then we can reformulate system \eqref{eq:quadConstr} by introducing additional variables $(x_0,\bm{x})\in \R^{n+1}$ as
\begin{subequations}\label{eq:quadReformulation}
\begin{align}&\sum_{i=0}^n x_i^2\leq 1,\; \bm x\circ(\ones-\bm z)=0,\;x_0(1-z_0)=0,\; \bm{z}\in Z\label{eq:quadReformulation_X}\\
&z_0=1,\;\sqrt{(\bm{y}'(\bm{R}/b) \bm{y})}\leq x_0,\;\sqrt{(D_{ii}/b)}|y_i|\leq x_i \text{ for }i=1,\dots,n, \label{eq:quadReformulation_conic}
\end{align}
\end{subequations}
where constraints \eqref{eq:quadReformulation_X} correspond precisely to $X$ and constraints \eqref{eq:quadReformulation_conic} are convex and SOCP-representable. Therefore, convex relaxations for system \eqref{eq:quadConstr} can be obtained by strengthening constraints \eqref{eq:quadReformulation_X} using $\conv(X)$. 

\paragraph{Robust optimization} Consider a robust optimization problem of the form
\begin{equation}\label{eq:robust}
\min_{\bm{y}\in Y}\max_{\bm{a}\in \mathcal{U}} \bm{a'y},
\end{equation}
where vector $\bm{y}$ are the decision variables, set $Y\subseteq \R^n$ is the (possibly non-convex) feasible region and set $\mathcal{U}\subseteq \R^n$ is an uncertainty set corresponding to the objective coefficients. Robust optimization \eqref{eq:robust} is a fundamental tool to tackle decision-making under uncertainty problems. Two popular choices for the uncertainty set $\mathcal{U}$, each with its own merits and disadvantages, are: the approach of \citet{ben2000robust}, where $\mathcal{U}$ is an ellipsoid; and the approach of \citet{bertsimas2004price}, where only a small subset of the coefficients $\bm{a}$ are allowed to change while satisfying box constraints. 

Thus, a natural uncertainty set inspired by the aforementioned two approaches allows few coefficients to change and imposes ellipsoidal constraint on the changing coefficients, that is,
set \begin{equation}\label{eq:uncertaintySet}\mathcal{U}\defeq \left\{\bm{a}\in \R^n: \exists \bm{(x,z)}\in \R^n\times \{0,1\}^n \text{ s.t. }\bm{a}=\bm{\tilde a}+\bm{x},\; \sum_{i=1}^n(d_ix_i)^2\leq b,\; \|\bm{z}\|_1\leq k,\; \bm x\circ(\ones-\bm z)=0\right\},\end{equation}
where $\bm{\tilde a}$ are the nominal values for the coefficients. The uncertainty set $\mathcal{U}$ is appropriate for example when changes in coefficients $\bm{a}$ are caused by rare events, and the change in the coefficients (when such changes occur) can be accurately modeled with a Gaussian distribution. Constraint $\|\bm{z}\|_1\leq k$ could be replaced by other constraints to capture more sophisticated relationships on the support of the perturbed coefficients.

Since set $\mathcal{U}$ is non-convex, solving \eqref{eq:robust} can be difficult and require sophisticated approaches \cite{borrero2021modeling}. Nonetheless, understanding $\conv(X)$ may lead to the possibility of using standard duality approaches to obtain deterministic counterparts of \eqref{eq:robust}. We further discuss this problem in \S\ref{sec:counterpart}.

\subsection{Perspective relaxation and outline} 
A closely related set to $X$ that is well understood in the literature is the mixed-integer epigraph of a separable quadratic function with indicators, that is, 
$X_{\text{epi}}\defeq \{(\bm x,\bm z,t)\in \R^n\times Z\times \R: \|\bm x\|_2^2\leq t,\;  \bm x\circ(\ones-\bm z)=0 \}.$ Its convex hull can be described via the \emph{perspective relaxation} $\text{cl conv}(X_{\text{epi}})=\{(\bm x,\bm z,t)\in \R^n\times \conv(Z)\times \R: \sum_{i=1}^nx_i^2/z_i\leq t\}$, see \cite{Ceria1999,Frangioni2006,akturk2009strong,Gunluk2010} for the case $Z=\{0,1\}^n$ and \cite{bacci2019new,wei2020convexification,wei2021ideal,wei2022convex,xie2020scalable} for cases with more general constraints. Thus, a natural convex relaxation for set $X$ is also given by the perspective relaxation
\begin{equation}
R_{\text{persp}}\defeq\{(\bm x,\bm z)\in \R^n\times \conv(Z): \sum_{i=1}^nx_i^2/z_i\leq 1\}. 
\end{equation}
However, it is unclear to what extent relaxation $R_{\text{persp}}$ coincides with $\conv(X)$: Are they the same? Is $R_{\text{persp}}$ ``necessary" to describe $\conv(X)$? Does the structure of $R_{\text{persp}}$ even ``matches" $\conv(X)$? Is $R_{\text{persp}}$ a strong relaxation? How can it be improved?

All these questions can be precisely answered for polyhedral sets: an inequality is necessary for a polyhedron if it is facet-defining; a relaxation matches the structure of a polyhedron if it is defined by a finite number of linear inequalities. However,
since $\conv(X)$ is in general non-polyhedral, it is unclear (to date) how to formally answer the aforementioned questions. Ideally, one would like to explicitly compute $\conv(X)$ and ``see" how well the set $R_{\text{persp}}$ matches this structure. Unfortunately, as we show in \S\ref{sec:NP}, optimization over set $X$ is NP-hard even when $Z=\{0,1\}^n$. Thus, an explicit computation of $\conv(X)$ is unlikely. This result immediately implies that $R_{\text{persp}}\neq conv(X)$, but does not provide insights into answering the remaining questions.

In this paper, we close this gap in the literature. In \S\ref{sec:hull} we characterize the structure of $\conv(X)$, and in particular we show that convexification of $X$ reduces to the convexification of a family of polyhedral sets. Interestingly, this family of polyhedra is well-studied in the literature. In \S\ref{sec:relaxations} we review how to obtain facet-defining inequalities, and we also show that $R_{\text{persp}}$ corresponds to using a strong nonlinear relaxation of these polyhedral sets. In \S\ref{sec:counterpart} we propose an approximate deterministic counterpart of the robust optimization problem \eqref{eq:robust} with discrete uncertainty \eqref{eq:uncertaintySet}, and in \S\ref{sec:computations} we present computations with this proposed formulation.

\section{NP-hardness}\label{sec:NP}
In this section we show that optimization of a linear function over set $X$ is NP-hard. This result indicates that a compact explicit computation of $\conv(X)$ is unlikely to be possible. 

Consider the optimization problem
\begin{subequations}\label{eq:optDiag}
\begin{align}
    \min_{\bm x,\bm z}\;& \bm a'\bm x+\bm c'\bm z\\
    \text{s.t.}\;&\|\bm x\|_2^2\leq 1\\
    &\bm x\circ(\ones-\bm z)=0\\
    &\bm x\in \R^n,\; \bm z\in Z.
\end{align}
\end{subequations}
\begin{proposition}\label{prop:NP}
Problem~\eqref{eq:optDiag} is NP-hard even if $Z=\{0,1\}^n$.
\end{proposition}
\begin{proof}
Consider problem \eqref{eq:optDiag} where vector $\bm z$ is fixed, and let $S=\{i\in [n]: z_i=1\}$ (assume $S\neq \emptyset$). Then, for this choice of $\bm z$, problem \eqref{eq:optDiag} reduces to 
\begin{subequations}\label{eq:optDiagFixed}
\begin{align}
    \epsilon_S=\min_{\bm x}\;&\sum_{i\in S}c_i+\sum_{i\in S}a_ix_i\\
    \text{s.t.}\;&\sum_{i\in S} x_i^2\leq 1\\
    &\bm x\in \R^S.
\end{align}
\end{subequations}

Since the Lagrangian dual of problem \eqref{eq:optDiagFixed} has no duality gap (as Slater condition holds), an optimal objective value $\epsilon_S$ can be computed as
\begin{align*}
    \epsilon_S&=\sum_{i\in S}c_i+\max_{\lambda\geq 0}\min_{\bm x\in \R^S}\sum_{i\in S}a_ix_i+\lambda\sum_{i\in S} x_i^2-\lambda\\
    &=\sum_{i\in S}c_i+\max_{\lambda\geq 0}-\frac{1}{4\lambda}\sum_{i\in S}a_i^2-\lambda\tag{$\because 2x_i^*=-a_i/\lambda$}\\
    &=\sum_{i\in S}c_i-\sqrt{\sum_{i\in S}a_i^2}\tag{$\because \lambda^*=\frac{1}{2}\sqrt{\sum_{i\in S}a_i^2}$}.
\end{align*}

In other words, the optimal vector $\bm z$ of \eqref{eq:optDiag} can be found by either setting $\bm z=\bm 0$ (with objective value $\epsilon_\emptyset=0$), or by solving the optimization problem
\begin{equation}\label{eq:discrete}
    \min_{\bm z\in Z}\sum_{i=1}^n c_iz_i-\sqrt{\sum_{i=1}^n a_i^2z_i}.
\end{equation}
Finally, as the partition problem can be reduced to problem \eqref{eq:discrete} with $Z=\{0,1\}^n$ (see \cite{ahmed2011maximizing}), problem \eqref{eq:optDiag} is NP-hard even in this case.
\end{proof}

\begin{remark}\label{rem:polySolvable}
If $\bm{c}=\bm{0}$ but there is a constraint of the form $\|\bm{z}\|_1=k$, then \eqref{eq:discrete} can be solved by sorting.  
Polynomial-time solvability of this case suggests that it may be possible to construct a convex relaxation that guarantees integrality of the solutions under these conditions. In other words, it may be possible to characterize the convex hull of the set
\begin{align*}
Y&=\left\{\bm{x}\in \R^n:\|\bm{x}\|_0\leq k,\; \|\bm{x}\|_2^2\leq 1\right\},
\end{align*}
where $\|\bm{x}\|_0=\sum_{i=1}^n \mathbbm{1}_{\{x_i\neq 0\}}$ is the cardinality of the support of $\bm{x}$. Indeed, set $Y$ is permutation-invariant, and its convex hull $\conv(Y)$ is described in \cite{kim2021convexification}, or projection of the perspective relaxation (i.e., $\conv(Y)=\textrm{proj}_{\bm x}(R_{persp})$). Note however that these relaxations are not ideal for $X$, i.e., solutions of linear optimization problems over $\conv(Y)$ do not coincide with the solutions of optimization problems over $X$ if $\bm{c}\neq \bm{0}$.
\qed
\end{remark}

\section{Structure of the convex hull}\label{sec:hull}
From Proposition~\ref{prop:NP}, we know that an explicit characterization of $\conv(X)$ is unlikely to be possible. In this section, we settle for a weaker structural result: in Theorem~\ref{theo:hull}, we state an explicit description of $\conv(X)$ that relies on the convex hulls of polyhedral sets. Naturally, describing these polyhedral sets is NP-hard as well; nonetheless, they are substantially easier to handle, thanks to the maturity of polyhedral theory. 

We first define the polyhedral sets that are key to characterizing $\conv(X)$. 
\begin{definition}
Given $\bm \alpha\in \R^n$, define sets $P_0(\bm \alpha),P(\bm \alpha)\subseteq \R^{2n}$ as 
\begin{align*}
	P_0(\bm{\alpha})\defeq&\left\{(\bm x,\bm z)\in \R^n\times Z:\sum_{i=1}^n | \alpha_i x_i|\leq  \sqrt{\sum_{i=1}^n  \alpha_i^2z_i}\right\}, \text{ and}\\
P(\bm{\alpha})\defeq&\conv\Big(P_0(\bm{\alpha})\Big).
\end{align*}
\end{definition}
Note that set $P(\bm{\alpha})$ is the convex hull of a union of a finite number of polytopes, one for each $\bm{z}\in Z$. Thus, $P(\bm{\alpha})$ is a polytope itself. We defer to \S\ref{sec:relaxations} the discussion on constructing relaxations of set $P(\bm{\alpha})$. As Proposition~\ref{prop:valid} below states, set $P(\bm{\alpha})$ is a relaxation of set $X$. 


\begin{proposition}[Validity]\label{prop:valid}
Set $\conv(X)\subseteq P(\bm{\alpha})$ for all $\bm{\alpha}\in \R^n$. 
\end{proposition}
\begin{proof}
	It suffices to show that $X\subseteq P_0(\bm{\alpha})$. 
Since $\bm x\circ(\ones-\bm z)=0$ and $\bm {z}\in\{0,1\}^n$, we must have $x_i=x_iz_i=x_i\sqrt{z}_i$ for all $i=1,\ldots,n$. 
Hence, we find that $$\sum_{i=1}^n |\alpha_ix_i|=\sum_{i=1}^k |(\alpha_i\sqrt{z}_i)x_i|\leq 
\sqrt{\sum_{i=1}^n\alpha_i^2z_i}\sqrt{\sum_{i=1}^nx_i^2}\leq \sqrt{\sum_{i=1}^n\alpha_i^2z_i},
$$
where the first inequality is due to H\"older's inequality, and the second one is because of $\sum_{i=1}^nx_i^2\leq 1$.
Hence, $(\bm{x},\bm{z})\in P_0(\bm{\alpha})\subseteq P(\bm{\alpha})$, concluding the proof. 
\end{proof}

Moreover, we now show how to use $P(\bm{\alpha})$ to construct an equivalent convex formulation of the NP-hard problem \eqref{eq:optDiag}. Note that in Proposition~\ref{prop:optD} below, we set $\bm{\alpha}=\bm{a}$.
\begin{proposition}[Optimality]\label{prop:optD}
	Problem \eqref{eq:optDiag} is equivalent to
	\begin{subequations}\label{eq:hull_diag}
	\begin{align}
	\min_{\bm x,\bm z}\;&\bm a'\bm x+\bm c'\bm z\\
	\text{s.t.}\;&(\bm x,\bm z)\in P(\bm a),\label{eq:hull_diag_constr}
	\end{align} 
	\end{subequations}
	that is, they both have the same optimal objective value and there exists an optimal solution of \eqref{eq:hull_diag} that is also optimal for \eqref{eq:optDiag}.
\end{proposition}
\begin{proof}
It suffices to show that problem \eqref{eq:optDiag} is equivalent to 
\begin{subequations}\label{eq:hull_diag_int}
	\begin{align}
		\min_{\bm x,\bm z}\;&\bm a'\bm x+\bm c'\bm z\\
		\text{s.t.}\;&(\bm x,\bm z)\in P_0(\bm a).
	\end{align} 
\end{subequations}

In any feasible solution of \eqref{eq:hull_diag_int}, we find that 
\begin{equation}\label{eq:optIneq}
\bm{a'x}\geq -\sum_{i=1}^n|a_ix_i|\geq -\sqrt{\sum_{i=1}^n a_i^2z_i},
\end{equation}
where the first inequality follows directly
from the definition of the absolute value and the second inequality follows from constraints $(\bm x,\bm z)\in P_0(\bm a)$. Moreover, both inequalities \eqref{eq:optIneq} hold at equality in an optimal solution, since otherwise, it is always possible to increase/decrease $x_i$ for some index $i$ without violating feasibility while improving the objective value. Thus, projecting out variables $\bm{x}$, problem \eqref{eq:hull_diag_int} reduces to \eqref{eq:discrete}, which, as shown in the proof of Proposition~\ref{prop:NP}, is equivalent to \eqref{eq:optDiag}.
\end{proof}

Propositions~\ref{prop:valid} and \ref{prop:optD} together come with an alternative representation of $\conv(X)$ which is expressed as intersections of sets $P(\bm{\alpha})$ for all $\bm{\alpha}\in \R^n$. 
\begin{theorem}\label{theo:hull}
	The convex hull of $X$
	can be described (with an infinite number of constraints, one for each $\bm{\alpha}\in \R^n$) as
	\begin{align}\label{eq:hullD}
	\conv(X)=\bigcap_{\bm{\alpha}\in \R^n}P(\bm \alpha)
	\end{align}
\end{theorem}
\begin{proof}
It is sufficient to show that the following two optimization problems 
\begin{align}
    &\min_{\bm{x},\bm{z}}\left\{\;\bm{a}'\bm{x}+\bm{c}'\bm{z} :(\bm x,\bm z)\in X\right\}\label{eq:optDiscrete}\\
    &\min_{\bm{x},\bm{z}}\left\{\;\bm{a}'\bm{x}+\bm{c}'\bm{z} :(\bm x,\bm z)\in \bigcap_{\bm{\alpha}\in \R^n}P(\bm \alpha)\right\}\label{eq:optContinuous}
\end{align}
are in fact equivalent. First, due to Proposition~\ref{prop:valid}, we find that \eqref{eq:optContinuous} is a relaxation of \eqref{eq:optDiscrete}. Second, the problem $\min_{\bm{x},\bm{z}}\left\{\;\bm{a}'\bm{x}+\bm{c}'\bm{z} :(\bm x,\bm z)\in P(\bm{a})\right\}$ is a further relaxation of \eqref{eq:optContinuous}, as it is obtained by dropping all the constraints but one. Third, due to Proposition~\ref{prop:optD}, this further relaxation is exact, and thus \eqref{eq:optContinuous} is exact as well. This concludes the proof.
\end{proof}

The description~\eqref{eq:hullD} of $\conv(X)$ can be highly nonlinear, since it involves an infinite number of constraints. However, the significance of Theorem~\ref{theo:hull} is that to understand $\conv(X)$ it suffices to study the polyhedral set $P(\bm{\alpha})$, which is arguably a simpler task due to advances in polyhedral theory, and since this set does not involve complementarity or other constraints linking the discrete and continuous variables. In \S\ref{sec:relaxations} we discuss how to obtain strong relaxation of $P(\bm{\alpha})$ in general. However, an alternative approach to obtain valid inequalities is to restrict the values of $\bm{\alpha}$, as the examples below show. 

\begin{example}\label{ex:bigM}
	Let $\bm{\alpha}=\bm{e}_i$ for some $i\in \{1,\ldots,n\}$, where $\bm{e}_i$ is the standard $i$-th basis vector of $\R^n$. In this case, \begin{align*}P(\bm{e}_i)&=\conv\left(\left\{(\bm x,\bm z)\in \R^n\times Z:|x_i|\leq \sqrt{z_i}\right\}\right)\\
	&=\left\{(\bm x,\bm z)\in \R^n\times \conv(Z):|x_i|\leq z_i\right\}.
	\end{align*} 
	 Thus, we find that big-M constraints \eqref{eq:bigM} are`` necessary" to describe $\conv(X)$. \qed
\end{example}

\begin{example}\label{ex:card}
	Suppose that $Z=\{\bm{z}\in \{0,1\}^n: \|\bm{z}\|_1= k\}$, and let $\bm{\alpha}=\ones$. In this case  \begin{align*}P(\ones)&=\conv\left(\left\{(\bm x,\bm z)\in \R^n\times \{0,1\}^n:\|\bm{x}\|_1\leq \sqrt{\|\bm{z}\|_1},\; \|\bm{z}\|_1= k\right\}\right)\\
		&=\left\{(\bm x,\bm z)\in \R^n\times [0,1]^n:\|\bm{x}\|_1\leq \sqrt{k},\;  \|\bm{z}\|_1= k\right\}.
	\end{align*} 
In particular we find that the inequality $\|\bm{x}\|_1\leq \sqrt{k}$, which was studied in \cite{dey2021using} in the context of sparse PCA, is ``necessary" to describe $\conv(X)$ in this case. 
	\qed
\end{example}

\ignore{
\begin{remark}
In a recent result, \citet{wei2022convex} showed that the convexification of the epigraph of an arbitrary convex quadratic function with indicator variables is given (in an extended formulation) by a single conic constraint (explicitly given) and polyhedral constraints. The results presented here are similar in the sense that the convexification of $X$ reduces to the description of the polyhedral sets $P(\bm{\alpha})$. However, the results are dissimilar in the sense that the description \eqref{eq:hullD} involves an \emph{infinite} number of polyhedral sets (albeit with similar structures), and thus is highly nonlinear even in an extended formulation.
\qed
\end{remark}
}

\section{Convex relaxations}\label{sec:relaxations}

This section discusses how to describe or approximate $P(\bm{\alpha})$. Interestingly, this family of polyhedra has already been studied in the literature. In \S\ref{sec:facets} we review existing results on the facial structure of $P(\bm{\alpha})$. In \S\ref{sec:natural} we study the natural \emph{nonlinear} relaxation of $P(\bm{\alpha})$, show that this relaxation is guaranteed to be strong, and establish links between this relaxation and the perspective relaxation $R_\text{persp}$. 

\subsection{Short review of relaxations via linear inequalities}\label{sec:facets}
We assume in this section that $Z=\{0,1\}^n$.
Given $\bm{\alpha}\in \R^n$, the facial structure of polyhedron $P(\bm{\alpha})$ was first studied in \cite{ahmed2011maximizing}, and the results were later refined in \cite{shi2022sequence}. We now review these results.
   
Define $N\defeq \{1,\dots,n\}$, and define the set function $g:2^N\to \R$ as $g(S)=\sqrt{\sum_{i\in S} \alpha_i^2}$. Since function $g$ is submodular, the submodular inequalities of \citet{nemhauser1978analysis} are valid for its hypograph. In particular, letting  $\rho_i(S)=g(S\cup\{i\})-g(S)$, the inequalities
\begin{subequations}\label{eq:submodular}
\begin{align}
\sum_{i=1}^n |\alpha_ix_{i}|&\leq g(S)-\sum_{i\in S} \rho_{i}(S\setminus\{i\})(1-z_i)+\sum_{i\in N\setminus S}\rho_i(\emptyset)z_i\quad&\forall S\subseteq N \label{eq:submodular1}\\
\sum_{i=1}^n |\alpha_ix_{i}|&\leq g(S)-\sum_{i\in S} \rho_{i}(N\setminus\{i\})(1-z_i)+\sum_{i\in N\setminus S}\rho_i(S)z_i\quad&\forall S\subseteq N \label{eq:submodular2}
\end{align}
\end{subequations}
are valid for $P(\bm{\alpha})$. However, coefficients $\rho_i(\emptyset)$ in \eqref{eq:submodular1} and $\rho_{i}(N\setminus\{i\})$ in \eqref{eq:submodular2} are not tight. Thus, inequalities \eqref{eq:submodular} are, in general, weak, and better inequalities can be obtained via lifting. Specifically, given $S\subseteq N$, the base inequality 
\begin{equation}\label{eq:base}\sum_{i=1}^n |\alpha_ix_{i}|\leq g(S)-\sum_{i\in S} \rho_{i}(S\setminus\{i\})(1-z_i)\end{equation}
is facet-defining for $\conv\left(\left\{(\bm x,\bm z)\in \R^n\times \{0,1\}^n:\sum_{i=1}^n |\alpha_ix_{i}|\leq \sqrt{\sum_{i=1}^n\alpha_i^2z_i},\; z_i=0 \;\forall i\in N\setminus S\right\}\right)$. Inequality~\eqref{eq:base} can then be lifted into a facet-defining inequality for $P(\bm{\alpha})$ through maximal lifting. In this case, lifting is sequence independent and the resulting inequality can be obtained in closed form, see \cite[Theorem 4]{shi2022sequence}. Similarly, inequality \eqref{eq:submodular2} can be improved through lifting, see \cite[Theorem 5]{shi2022sequence}. While the inequalities discussed here are facet-defining for the case $Z=\{0,1\}^n$, they may be weaker for the case with more general constraints. Nonetheless, we point out that strong valid inequalities have also been proposed for the case where $Z$ is defined by a knapsack constraint, see~\cite{yu2017maximizing}.

\subsection{Natural convex relaxation}\label{sec:natural}

Consider the natural nonlinear relaxation of $P(\bm{\alpha})$, obtained by simply dropping the integrality constraints on variables $\bm{z}$:
\begin{align*}
C(\bm{\alpha})\defeq&\left\{(\bm{x},\bm{z})\in \R^n\times \conv(Z):\sum_{i=1}^n |\alpha_ix_{i}|\leq \sqrt{\sum_{i=1}^n\alpha_i^2z_i}\right\}.
\end{align*}
While $C(\bm{\alpha})$ is hard to compute in general as it involves computing the convex hull of the feasible region $Z$, it can be obtained easily for example if $Z=\{0,1\}^n$, $Z=\{z\in \{0,1\}^n: \sum_{i=1}^n z_i\leq k\}$ for some $k\in \Z_+$, or more generally if the constraints defining $Z$ are totally unimodular. Moreover, the nonlinear constraint defining $C(\bm{\alpha})$ is SOCP-representable. Thus, this continuous relaxation can be used with many off-the-shelf solvers. 

Optimization over relaxation $C(\bm{\alpha})$ has also been studied in the literature \cite{atamturk2017maximizing}. Specifically, consider the convex relaxation of the problem \eqref{eq:hull_diag} given by 
\begin{subequations}\label{eq:opt_relax}
	\begin{align}
	\bar\zeta=\min_{\bm x,\bm z}\;&\bm a'\bm x+\bm c'\bm z\\
	\text{s.t.}\;&(\bm x,\bm z)\in C(\bm a),
	\end{align} 
\end{subequations}

\begin{proposition}\label{prop:relax} There exists an optimal solution $(\bm{\bar x},\bm{\bar z})$ of \eqref{eq:opt_relax} where $\bm{\bar z}$ lies on an edge of $\conv(Z)$. Moreover, if $\bm{c}'\bm{z}\leq 0$ for all $\bm{z}\in Z$, then $(4/5)\bar \zeta\geq \zeta^*\geq (5/4)\bar \zeta_r$, where $\zeta^*$ is the optimal objective value of problem \eqref{eq:hull_diag} --equivalently, problem \eqref{eq:optDiag}--, and $\zeta_r$ is the objective value of the feasible solution obtained by rounding $\bm{\bar z}$ to the best of the two extreme points of $\conv(Z)$ defining the edge where it lies. 
\end{proposition}
In other words, Proposition~\ref{prop:relax} states that the solution of \eqref{eq:opt_relax} is ``close" to integral (e.g., if $Z=\{0,1\}^n$, then $\bm{\bar z}$ has at most one fractional coordinate), that its associated objective value is similar to the optimal objective value of the mixed-integer problem, and that rounding of this solution yields a constant factor approximation algorithm under mild conditions. 

\begin{proof}[Proof of Proposition~\ref{prop:relax}]
	Projecting out variables $\bm{x}$ exactly the same as the proof of Proposition~\ref{prop:optD}, we find that problem \eqref{eq:opt_relax} simplifies to 
\begin{equation}\label{eq:convex}
\min_{\bm z\in \conv(Z)}\sum_{i=1}^n c_iz_i-\sqrt{\sum_{i=1}^n a_i^2z_i}.
\end{equation}
This particular continuous relaxation of the discrete problem with feasible region $\bm{z}\in Z$ was studied in \cite{atamturk2017maximizing}, and all the results in the proposition follow directly from that paper.  
\end{proof}

Now consider the relaxation of $\conv(X)$, as defined in \eqref{eq:hullD}, obtained by replacing polyhedra $P(\bm{\alpha})$ with their nonlinear relaxations $C(\bm{\alpha})$:
	\begin{align}
\bar C\defeq\bigcap_{\bm{\alpha}\in \R^n}C(\bm \alpha)=\Big\{(\bm{x},\bm{z})\in \R^{2n}:  (\bm{x},\bm{z})\in C(\bm \alpha),\; \forall \bm{\alpha}\in \R^n\Big\}.
\end{align}
Proposition~\ref{prop:perspective} below states that the relaxation $\bar C$ is in fact equivalent to the perspective relaxation.

\begin{proposition}\label{prop:perspective} 
	$\bar C=R_{\text{persp}}.$
\end{proposition}
\begin{proof}
Note that the set $\bar C$ can be described with constraint $\bm{z}\in \conv(Z)$ and the single nonlinear constraint
		\begin{equation}\label{eq:constrainedform3.5}
	0\geq \max_{\bm{\alpha}\in \R^n}\sum_{i=1}^n | \alpha_i x_i|- \sqrt{\sum_{i=1}^n  \alpha_i^2z_i}.
	\end{equation}
	Since the function in \eqref{eq:constrainedform3.5} is positively homogeneous in $\bm{\alpha}$, it follows that either the optimization problem in unbounded (and the constrained is violated), or the optimization problem is bounded (and the constraint is satisfied). Finally, a characterization on whether this problem is bounded or not can be found in \cite[Proposition 2]{gomez2021strong}: problem \eqref{eq:constrainedform3.5} is unbounded if and only if $\sum_{i=1}^n x_i^2/z_i>1$. Thus, concluding the proof.
	
	\ignore{
	Now suppose that $\sum_{i=1}^n\frac{x_i^2}{z_i}>1$. In that case, consider the solutions given by $\alpha_i=\kappa \frac{x_i}{z_i}$ for $\kappa> 0$. In that case, we find that 
	\begin{align*}
	\sum_{i=1}^n | \alpha_i x_i|- \sqrt{\sum_{i=1}^n  \alpha_i^2z_i}&=\kappa\sum_{i=1}^n \frac{x_i^2}{z_i}- \kappa\sqrt{\sum_{i=1}^n\frac{x_i^2}{z_i}}\\
		&=\kappa\sqrt{\sum_{i=1}^n \frac{x_i^2}{z_i}}\underbrace{\left(\sqrt{\sum_{i=1}^n \frac{x_i^2}{z_i}}-1\right)}_{>0}\\
		&> 0.
		\end{align*}
In particular, we find that if $(\bm{x},\bm{z})\not\in \Big\{(\bm{x},\bm{z})\in \R^{n}\times [0,1]^n: \sum_{i=1}^n \frac{x_i^2}{z_i}\leq 1\Big\}$, then $(\bm{x},\bm{z})\not\in \bar C$, concluding the proof. }

\end{proof}

\begin{remark}
Observe that the big-M constraints \eqref{eq:bigM} are not implied by relaxation $\bar C=R_{\text{persp}}$. Although these inequalities are not hugely beneficial (in light of Proposition~\ref{prop:relax}), they should be still added to the relaxation due to their simplicity and because they are required to describe $\conv(X)$ (as shown in Example~\ref{ex:bigM}).\qed
\end{remark}
\ignore{
\begin{remark}\label{ex:card}
	Suppose that $Z=\{\bm{z}\in \{0,1\}^n: \|\bm{z}\|_1= k\}$, and let $\bm{\alpha}=\ones$. In this case  \begin{align*}P(\ones)&=\conv\left(\left\{\bm{z}\in \{0,1\}^n,\bm{w}\in \R_+^n:\frac{1}{2}\sum_{i=1}^n w_{i}\leq \sqrt{\|\bm{z}\|_1},\; \|\bm{z}\|_1= k\right\}\right)\\
	&=\left\{\bm{z}\in [0,1]^n,\bm{w}\in \R_+^n:\frac{1}{2}\sum_{i=1}^n w_{i}\leq \sqrt{k},\;  \|\bm{z}\|_1= k\right\}\\
	&=C(\ones),
	\end{align*} 
	that is, the natural continuous relaxation is exact in this case. 
	Moreover, we find that  
	\begin{align}
	    &	\exists \bm{w}\in \R_+^n \text{ s.t. }\sum_{i=1}^n\frac{x_i^2}{w_i}-\frac{1}{4}\sum_{i=1}^n w_i\leq 0,\; (\bm{z},\bm{w})\in C(\bm \ones)\notag\\
	    \Leftrightarrow &\sum_{i=1}^n |x_i|\leq \sqrt{k},\label{eq:absolute}
	\end{align}
	where the equivalence is found by repeating all the steps of the proof of Proposition~\ref{prop:perspective} leading to inequality~\eqref{eq:constrainedform3}, and replacing $\sqrt{\sum_{i=1}^n\alpha_i^2z_i}=\sqrt{k}$. In particular, we find that inequality \eqref{eq:absolute}, which was proposed in \cite{dey2021using} in the context of sparse PCA, is ``necessary" to describe $\conv(X)$, and is also implied by the natural continuous relaxation $\bar C=R_{\text{persp}}$. 
	\qed
\end{remark}
}

\section{Approximate robust counterpart}\label{sec:counterpart}

We now turn our attention to the robust optimization problem~\eqref{eq:robust} with uncertainty set \eqref{eq:uncertaintySet}, discussed in \S\ref{sec:apps}. 
Instead of solving \eqref{eq:robust} directly, which is difficult due to the discrete uncertainty set, we propose to solve instead the perspective approximation
\begin{equation}\label{eq:approximation}
	\xi=\min_{\bm{y}\in Y}\bm{\tilde a'y}+\max_{(\bm{x},\bm{z})\in \R^n\times [0,1]^n}\left\{\bm{x'y}:\sum_{i=1}^n(d_ix_i)^2/z_i\leq b,\;\sum_{i=1}^n z_i\leq k\right\}.
\end{equation}
Since we relaxed the inner maximization problem, it follows that \eqref{eq:approximation} is a conservative approximation of \eqref{eq:robust}. Moreover, since $\bm{z}$ does not appear in the objective of the inner maximization problem, the condition of Proposition~\ref{prop:relax} is satisfied: for any fixed $\bm{y}$ the objective value of the inner maximization problem in \eqref{eq:approximation} is at most $5/4$ times the corresponding objective value
in \eqref{eq:robust}. Thus, if $\bm{\tilde a'y}\geq 0$ for all $\bm{y}\in Y$, then solving \eqref{eq:approximation} results in a $1.25$-approximation algorithm for \eqref{eq:robust}. 
We now derive a conic-quadratic formulation of problem \eqref{eq:approximation}. 

\begin{proposition}\label{prop:counterpart}
	Given $\bm{y}\in \R^n$, let $\{(1),(2),\dots,(n)\}$ be a permutation of $\{1,\dots,n\}$ such that $(y_{(1)}/d_{(1)})^2\geq (y_{(2)}/d_{(2)})^2\geq \dots \geq (y_{(n)}/d_{(n)})^2$, and let $s(\bm{y})=\sum_{i=1}^k (y_{(i)}/d_{(i)})^2$ be the sum of the largest $k$ such values. Then problem \eqref{eq:approximation} is equivalent to $$\xi=\min_{\bm{y}\in Y}\bm{\tilde a'y}+\sqrt{bs(\bm{y})}.$$
	Moreover, 
	this problem can be reformulated as the SOCP 
	\begin{subequations}\label{eq:robustSocp}
	\begin{align}
		\min_{\bm{y},\bm{t},\lambda,\mu}\;&\bm{\tilde a'y}+\lambda b+\mu k+\sum_{i=1}^n t_i\\
		\text{s.t.}\;& (y_i/d_i)^2\leq 4(t_i+\mu)\lambda&i=1,\dots,n\label{eq:robustSocp_rotated}\\
		&\bm{y}\in Y\\
		&\bm{t}\in \R_+^n,\; \lambda\in \R_+,\; \mu\in \R_+. 
			\end{align}
		\end{subequations}
\end{proposition}

Observe that since both $\lambda\geq 0$ and $t_i+\mu\geq 0$, \eqref{eq:robustSocp_rotated} are rotated cone constraints and thus \eqref{eq:robustSocp} is indeed SOCP-representable (provided that $Y$ is). 
The derivation of Proposition~\ref{prop:counterpart} is based on the following Fenchel duality result used in \cite{atamturk2020safe}.
\begin{lemma}[Fenchel dual]\label{lem:fenchel}
	For any $x\in \R$ and $0\leq z\leq 1$,
	\begin{equation*}\frac{x^2}{z}=\max_{p\in \R}px-\frac{p^2}{4}z.\end{equation*}
\end{lemma}
\begin{proof}
	If $x=z=0$, then both sides of the equality are $0$. If $z=0$ and $x\neq 0$, then both sides are equal to $+\infty$. Otherwise, an optimal solution of the maximization problem is $p^*=2\frac{x}{z}$, and the corresponding objective value is~$\frac{x^2}{z}$. 
\end{proof}

\begin{proof}[Proof of Proposition~\ref{prop:counterpart}] We find that
\begin{align*}
\xi&=\min_{\substack{\bm{y}\in Y\\\lambda,\mu\in \R_+}}\bm{\tilde a'y}+\lambda b+\mu k+\max_{(\bm{x},\bm{z})\in \R^n\times [0,1]^n}\left\{\bm{y'x}-\lambda\sum_{i=1}^n(d_ix_i)^2/z_i-\mu\sum_{i=1}^n z_i\right\}\\
&\tag{$\because$ Slater condition holds and strong duality of Lagrangian relaxation}\\
&=\min_{\substack{\bm{y}\in Y\\\lambda,\mu\in \R_+,\bm{p}\in \R^n}}\bm{\tilde a'y}+\lambda b+\mu k+\max_{(\bm{x},\bm{z})\in \R^n\times [0,1]^n}\left\{\sum_{i=1}^n(y_i-\lambda d_ip_i)x_i+\sum_{i=1}^n(0.25\lambda p_i^2-\mu)z_i\right\}\tag{$\because$ Lemma~\ref{lem:fenchel} and Sion's minimax theorem}\\
&=\min_{\substack{\bm{y}\in Y\\\lambda,\mu\in \R_+}}\bm{\tilde a'y}+\lambda b+\mu k+\max_{\bm{z}\in  [0,1]^n}\left\{\sum_{i=1}^n\Big(0.25\frac{(y_i/d_i)^2}{\lambda}-\mu\Big)z_i\right\}\tag{$\because p_i^*=y_i/(\lambda d_i)$}\\
&=\min_{\substack{\bm{y}\in Y\\\lambda,\mu\in \R_+}}\bm{\tilde a'y}+\lambda b+\mu k+\sum_{i=1}^n\max\left\{0,\frac{(y_i/d_i)^2}{\lambda}-\mu\right\}\tag{$\because z_i^*=\mathbbm{1}_{\{0.25(y_i/d_i)^2>\lambda\mu\}}$}.
\end{align*}
The formulation above corresponds directly to the SOCP formulation \eqref{eq:robustSocp}. We now continue projecting out variables  to recover the explicit form in the original space of variables:
\begin{align*}
\xi&=\min_{\substack{\bm{y}\in Y\\\lambda,\mu\in \R_+}}\bm{\tilde a'y}+\lambda b+\mu k+\frac{1}{\lambda}\sum_{i=1}^n\max\left\{0,0.25(y_i/d_i)^2-\lambda\mu\right\}\\
&=\min_{\substack{\bm{y}\in Y\\\lambda,\gamma\in \R_+}}\bm{\tilde a'y}+\lambda b+\frac{1}{\lambda}\left(\gamma k+\sum_{i=1}^n\max\left\{0,0.25(y_i/d_i)^2-\gamma\right\}\right)\tag{$\gamma\defeq\lambda\mu$}\\
&=\min_{\substack{\bm{y}\in Y\\\gamma\in \R_+}}\bm{\tilde a'y}+2\sqrt{b}\sqrt{\gamma k+\sum_{i=1}^n\max\left\{0,0.25(y_i/d_i)^2-\gamma\right\}}\tag{$\lambda^*=\sqrt{\frac{\gamma k+\sum_{i=1}^n\max\left\{0,0.25(y_i/d_i)^2/-\gamma\right\}}{b}}$}.
\end{align*}
Finally, for any fixed $\bm{y}$, an optimal value of $\gamma$ is given by $(k+1)$-largest value of $0.25(y_i/d_i)^2$, i.e., $\gamma^*=0.25(y_{(k+1)}/d_{(k+1)})^2$, concluding the proof.
\end{proof}

\section{Computations}\label{sec:computations}

According to the results of \S\ref{sec:relaxations}, the perspective is a simple relaxation that is guaranteed to be strong (Proposition~\ref{prop:relax}). Thus, we suggest its use in practice. Note that if set $X$ appears directly in an optimization problem (e.g., the first two applications discussed in \S\ref{sec:apps}),
the perspective is arguably already the state-of-the-art relaxation -- thus we omit computations for those cases. However, we illustrate its application to the robust optimization problem \eqref{eq:robust} with uncertainty set \eqref{eq:uncertaintySet}.
In particular, we consider a simple portfolio optimization problem with $Y=\{\bm{y}\in \R^n: \sum_{i=1}^n y_i=1,\; \bm{y}\geq 0\}.$

\subsection{Methods} We compare three conservative approximations of \eqref{eq:robust} -- the first two are based on commonly used methods in the literature.

\paragraph{Budgeted uncertainty} This approach, inspired by \cite{bertsimas2004price}, replaces the ellipsoidal constraint with simple bound constraints and solves instead
 \begin{equation*}
 	\min_{\bm{y}\in Y}\bm{\tilde a'y}+\max_{(\bm{x},\bm{z})\in \R^n\times \{0,1\}^n}\left\{\bm{x'y}:|x_i|\leq \sqrt{b}/d_i,\;\sum_{i=1}^n z_i\leq k,\;\bm x\circ(\ones-\bm z)=0\right\}.
 \end{equation*}
This optimization problem can be reformulated as the linear optimization \cite{bertsimas2004price}
\begin{align*}
		\min_{\bm{y},\bm{t},\mu}\;&\bm{\tilde a'y}+b\mu+\sum_{i=1}^n t_i\\
		\text{s.t.}\;&(\sqrt{b}/d_i) |y_i|\leq \mu+t_i\qquad i=1,\dots,n\\
		&\bm{y}\in Y,\; \bm{t}\in \R_+^n,\;\mu\in \R_+.
\end{align*}
Note that $\bm{y}\geq 0$ in our experiments. Thus, we replace $|y_i|$  with $y_i$ in all constraints.
\paragraph{Ellipsoidal uncertainty} This approach, inspired by \cite{ben2000robust}, ignores the cardinality constraint and solves instead
\begin{equation*}
	\min_{\bm{y}\in Y}\bm{\tilde a'y}+\max_{\bm{x}\in \R^n}\left\{\bm{x'y}:\sum_{i=1}^n(d_ix_i)^2\leq b\right\}.
\end{equation*}
This optimization problem can be reformulated as the SOCP \cite{ben2000robust}
\begin{align*}
	\min_{\bm{y},\bm{t},\mu}\;&\bm{\tilde a'y}+\sqrt{b}\cdot\sqrt{\sum_{i=1}^n (y_i/d_i)^2}\\
	\text{s.t.}\;
	&\bm{y}\in Y,\; \bm{t}\in \R_+^n,\;\mu\in \R_+.
\end{align*}

\paragraph{Perspective approximation} The approach we propose, described in \S\ref{sec:counterpart}. 

\subsection{Results} We set $n=200$ in our computations, and we set $k\in \{5,10,20\}$ and $b\in \{5,10,20\}$ in our computations. Each entry of $\bm{a}$ and $\bm{d}$ is drawn from an uniform distribution on the interval $[0,1]$ -- under these conditions, since $\bm{y}\geq 0$ and $\bm{a}\geq 0$, then $\bm{a'y}\geq 0$ for all $\bm{y}\in Y$ and the perspective approximation is a 1.25 approximation algorithm for \eqref{eq:robust}. All optimization problems are solved using CPLEX 12.8 with the default settings, in a laptop with  Intel Core i7-8550U CPU and 16 GB RAM. Solution times for all methods are less than 0.1 seconds in all cases. 

 For each combination of parameters $(b,k)$, we generate 10 instances and record for each method: the nominal objective value $\bm{\tilde a'y^*}$, where $\bm{y^*}$ is the solution produced; and the worst-case realization given by
\begin{equation}\label{eq:mio} \bm{\tilde a'y^*}+\max_{(\bm{x},\bm{z})\in \R^n\times \{0,1\}^n}\left\{\bm{x'y^*}:\sum_{i=1}^n(d_ix_i)^2/z_i\leq b,\;\sum_{i=1}^n z_i\leq k,\; \bm x\circ(\ones-\bm z)=0\right\}.\end{equation}
Note that computing the worst-case realization requires solving a mixed-integer optimization problem. However, since the perspective reformulation results in a strong relaxation and $n=200$ is not too large, problem~\eqref{eq:mio} can be comfortably solved to optimality using CPLEX. Figure~\ref{fig:comp} presents the results, showing the nominal objective value and worst-case realization for each combination of parameters and each instance. 

\begin{figure}[!hp]
	\centering
	\subfloat[$k=5,b=5$]{\includegraphics[width=0.33\textwidth,trim={11cm 6cm 11cm 5cm},clip]{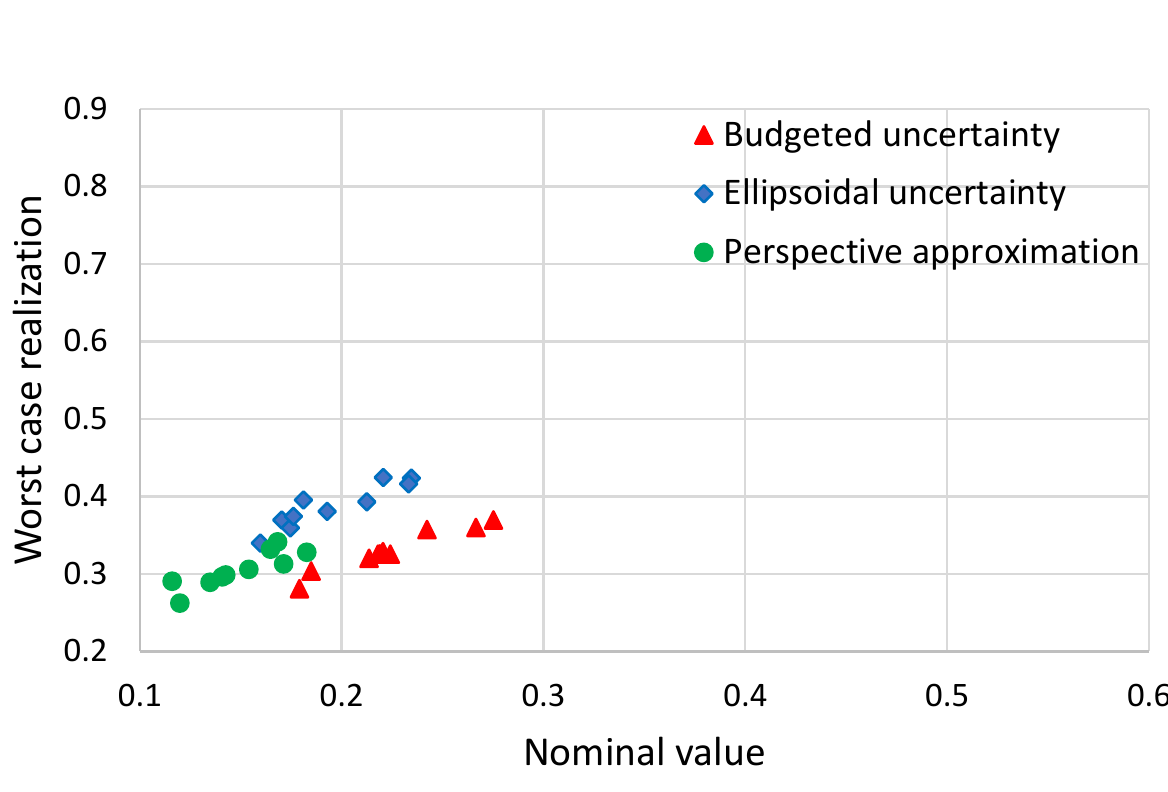}}\hfill\subfloat[$k=5,b=10$]{\includegraphics[width=0.33\textwidth,trim={11cm 6cm 11cm 5cm},clip]{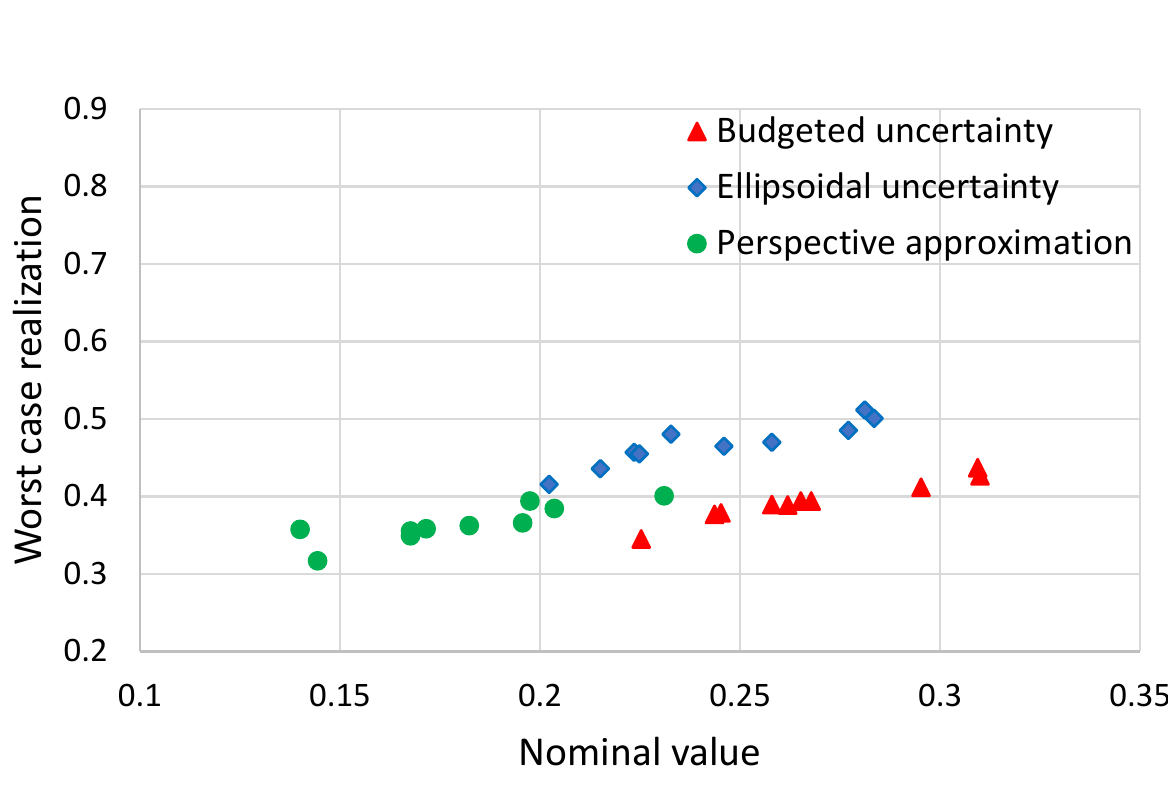}}\hfill	\subfloat[$k=5,b=20$]{\includegraphics[width=0.33\textwidth,trim={11cm 6cm 11cm 5cm},clip]{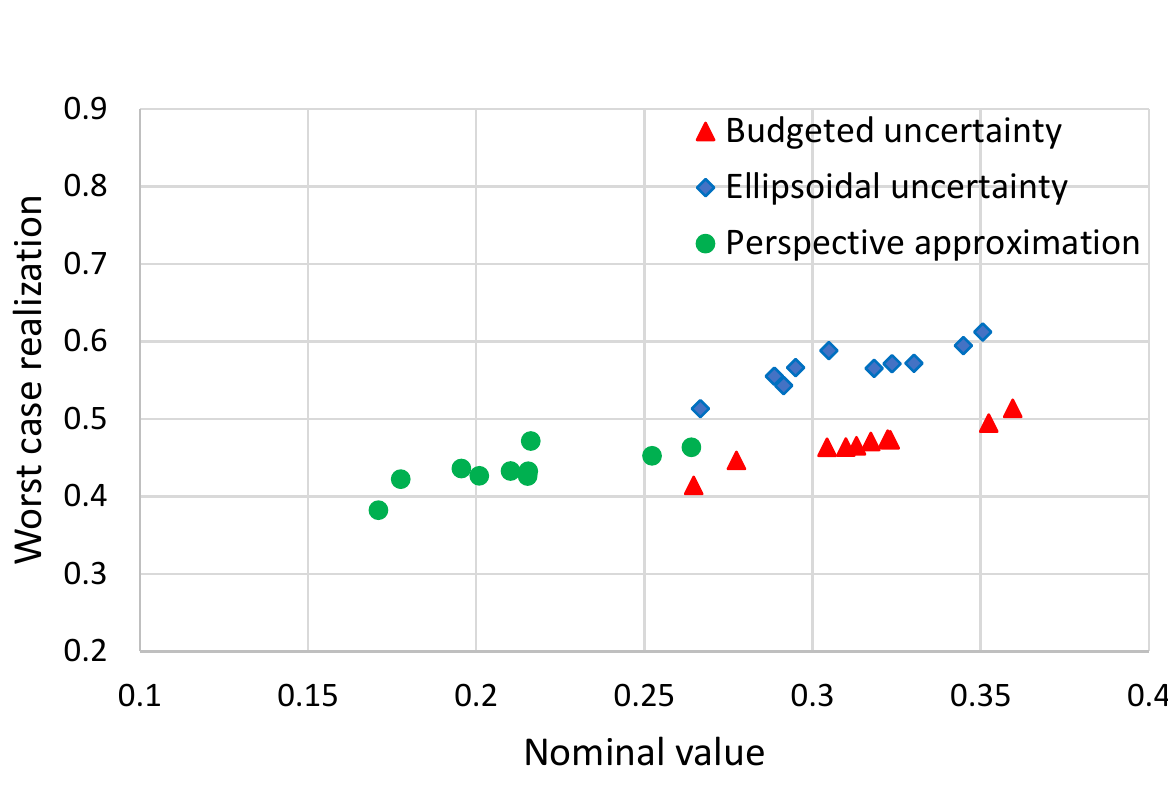}}\hfill\newline \vskip 5mm \subfloat[$k=10,b=5$]{\includegraphics[width=0.33\textwidth,trim={11cm 6cm 11cm 5cm},clip]{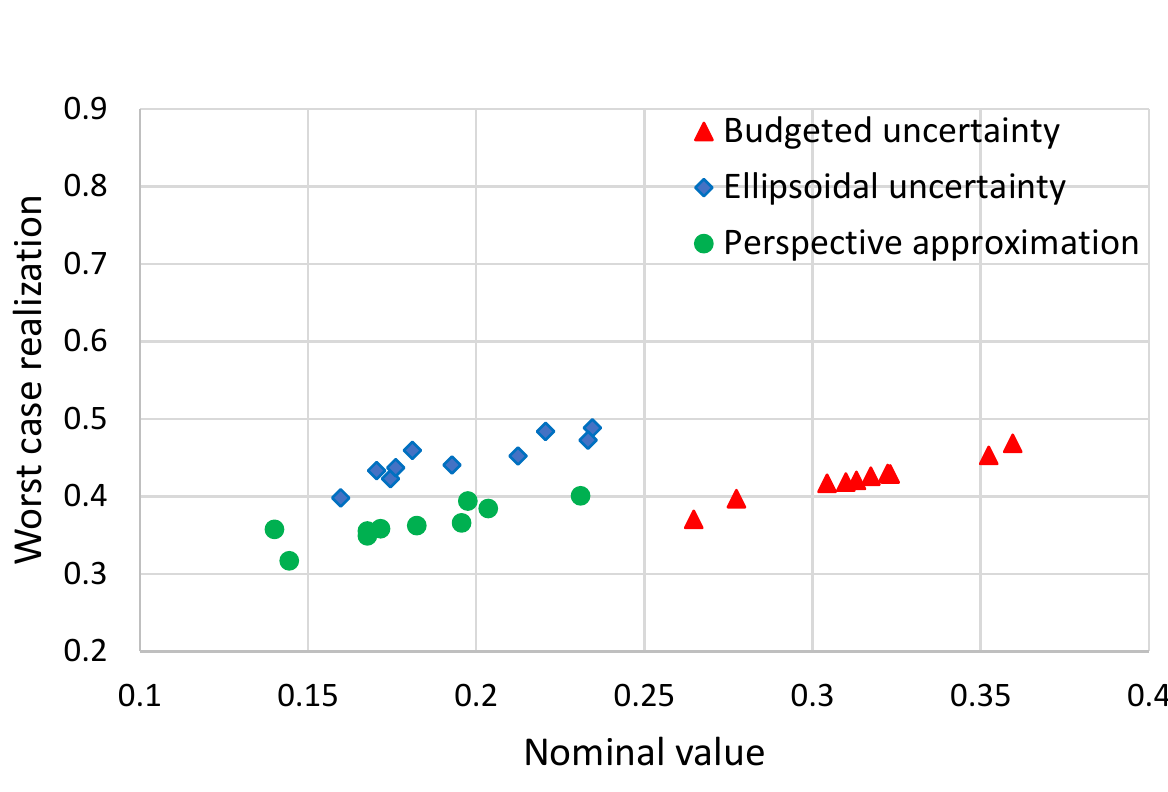}}\hfill\subfloat[$k=10,b=10$]{\includegraphics[width=0.33\textwidth,trim={11cm 6cm 11cm 5cm},clip]{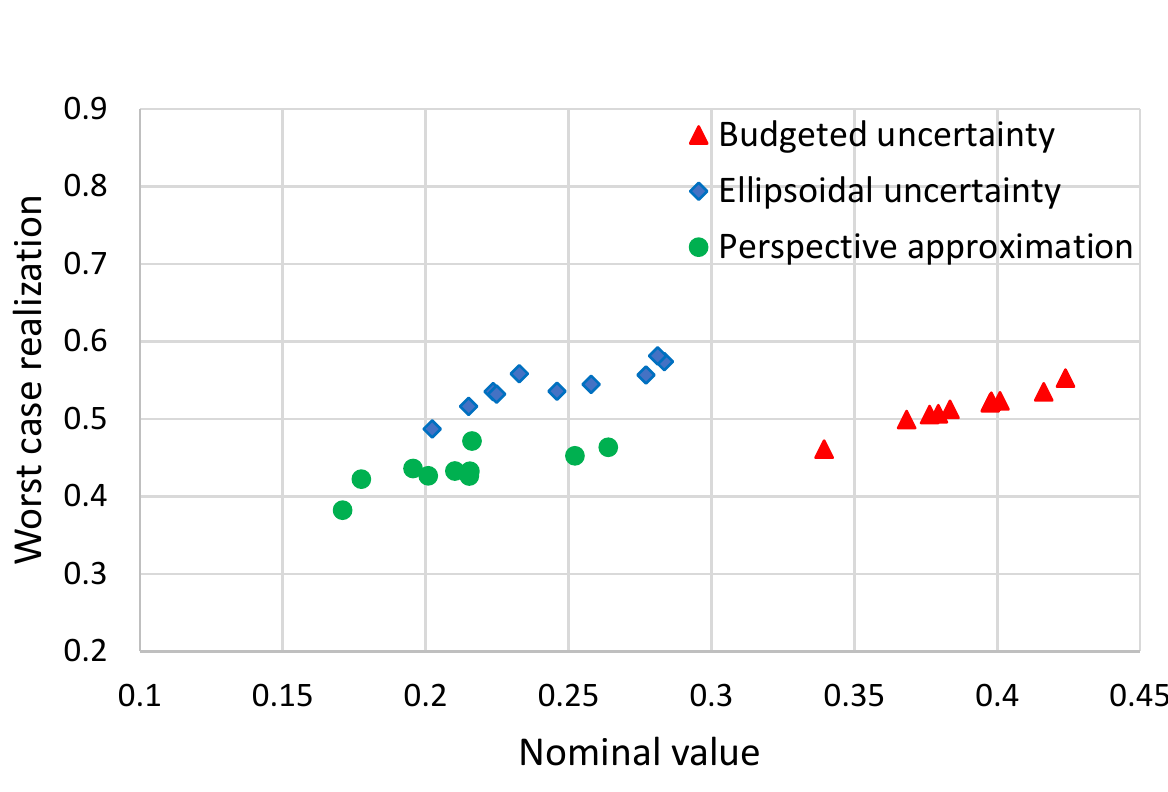}}\hfill\subfloat[$k=10,b=20$]{\includegraphics[width=0.33\textwidth,trim={11cm 6cm 11cm 5cm},clip]{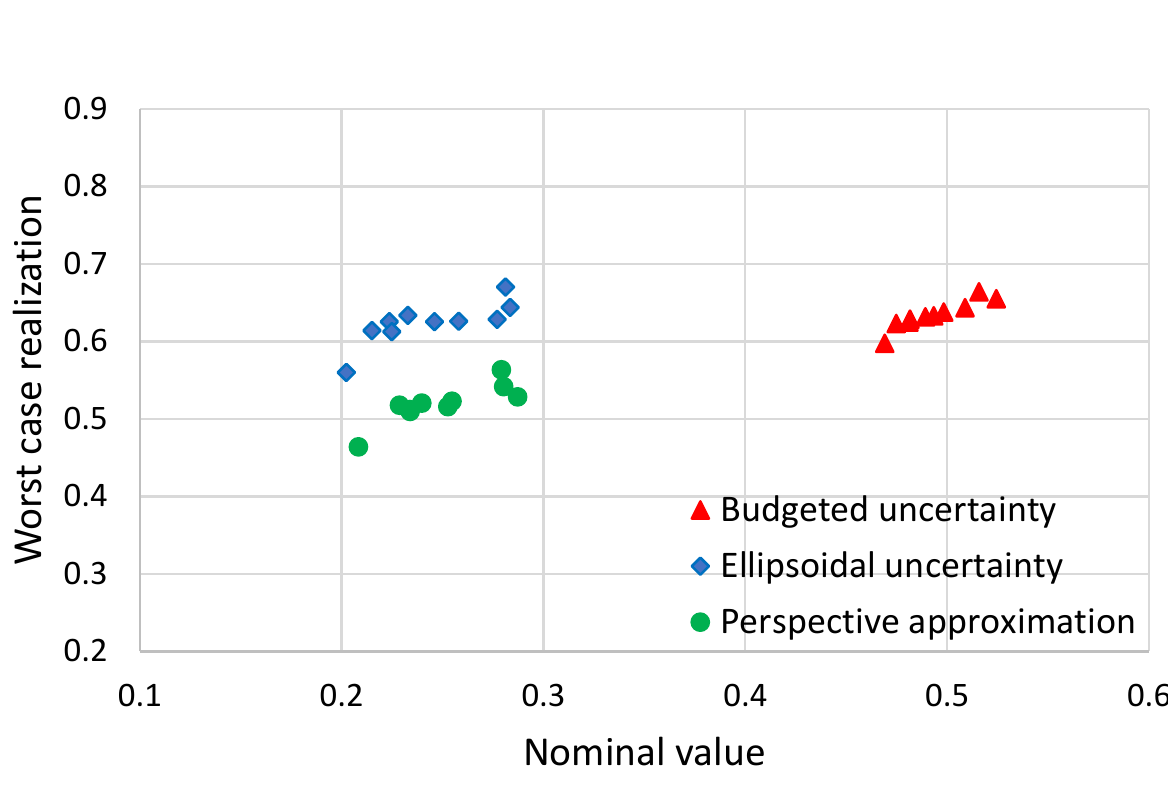}}\hfill\newline \vskip 5mm
	\subfloat[$k=20,b=5$]{\includegraphics[width=0.33\textwidth,trim={11cm 6cm 11cm 5cm},clip]{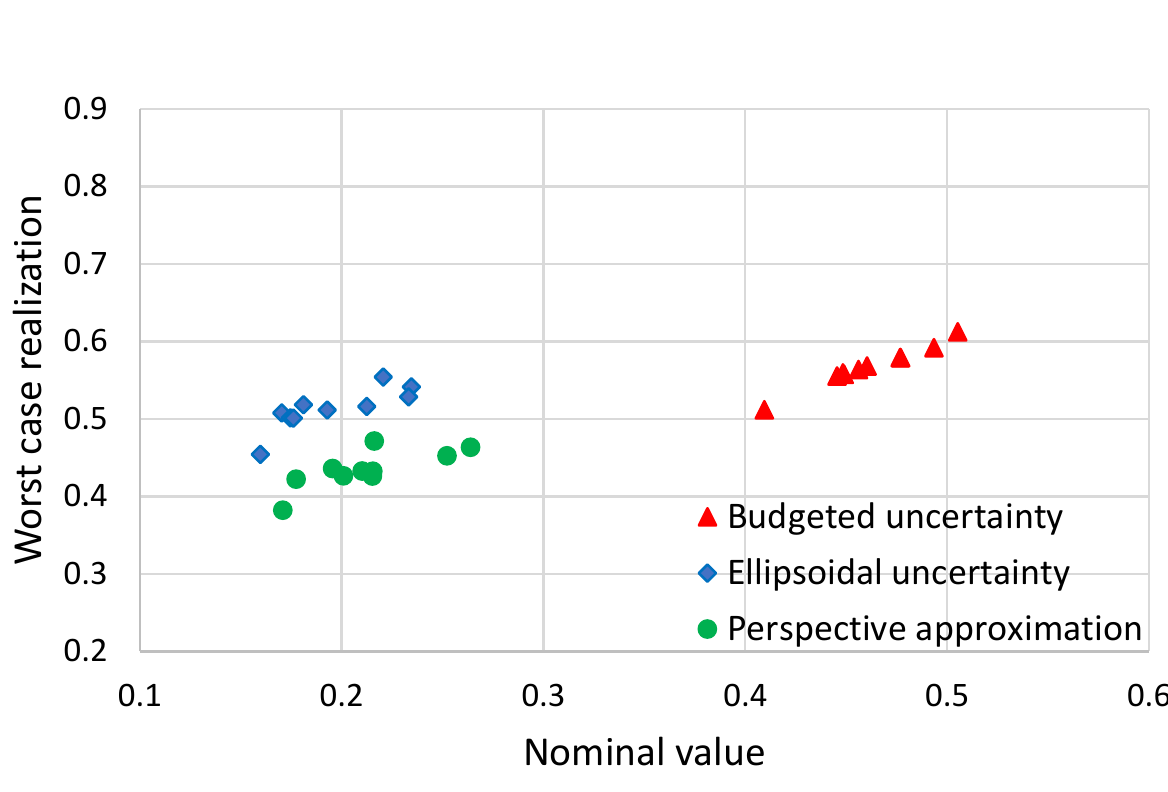}}\hfill\subfloat[$k=20,b=10$]{\includegraphics[width=0.33\textwidth,trim={11cm 6cm 11cm 5cm},clip]{./k20b10.pdf}}\hfill\subfloat[$k=20,b=20$]{\includegraphics[width=0.33\textwidth,trim={11cm 6cm 11cm 5cm},clip]{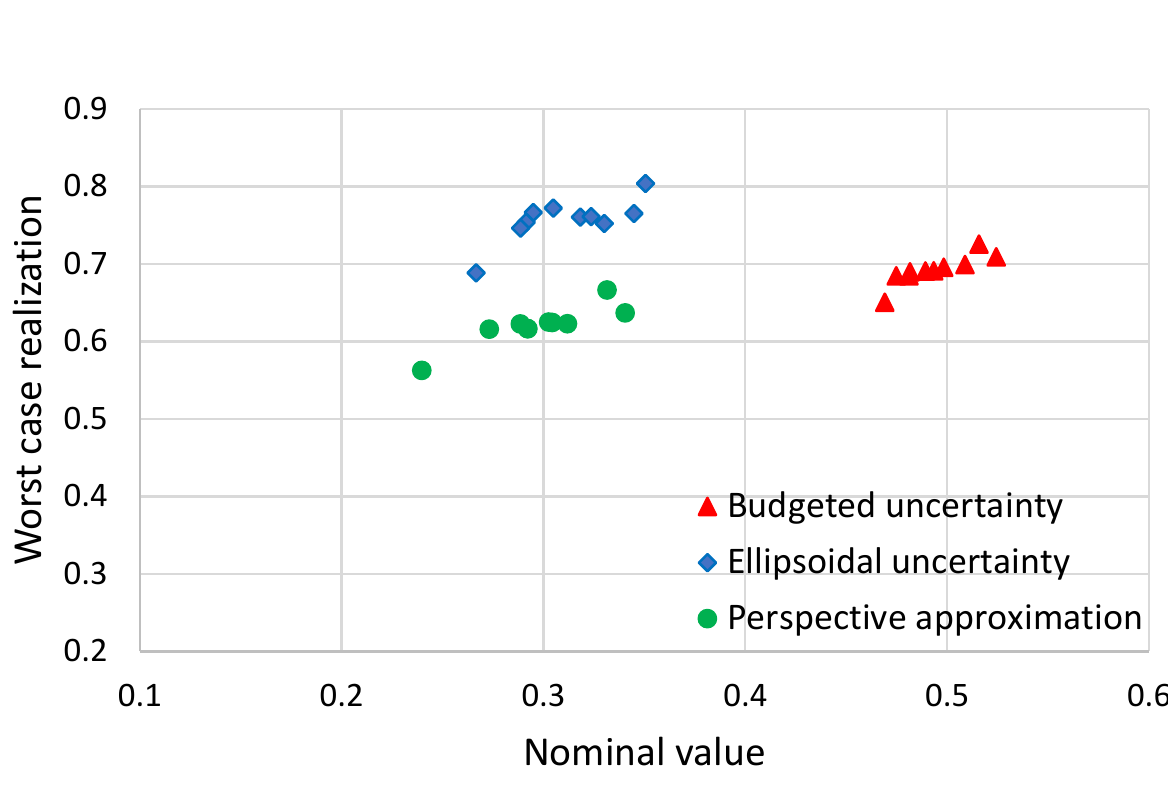}}\hfill
	\caption{\small Nominal value versus worst-case realization for different cardinality and budget parameters. The \textbf{budgeted uncertainty} approach (red triangles) typically yields solutions with large nominal values, particularly for large values of $k$. The \textbf{ellipsoidal uncertainty} approach (blue rhombuses) often results in good nominal values (particularly for large $k$), but the worst-case realizations are large. The \textbf{perspective approximation} (red circles) always results in the best worst-case realizations, and often in the best nominal values.}
	\label{fig:comp}
\end{figure}

We observe that the budgeted uncertainty approach consistently has the worst nominal performance, although it tends to be better in terms of robustness than the ellipsoidal uncertainty. The perspective approximation results in the ``best" worse-case realizations for all the combinations of parameters. It also results in the best solutions in terms of the nominal values, except for the case with $k=20$ and $b=5$ (where the ellipsoidal uncertainty has slightly better nominal performance). \emph{Thus, in our experiments, we can conclude that the perspective approximation is the best approach, delivering the most reliable solutions without affecting (and in most cases improving) the nominal performance.} 

\section{Conclusion}
We characterized the structure of the set $\conv(X)$, established links between the convexification of this set and convexification of polyhedral sets, and studied the strength of the perspective relaxation $R_\text{persp}$.
On the one hand, we showed in this paper that the perspective reformulation is insufficient to describe $\conv(X)$, and that $R_\text{persp}$ does not even match the structure of $\conv(X)$: using the perspective reformulation to approximate $\conv(X)$ is akin to using a nonlinear relaxation to approximate a polyhedral set, see Proposition~\ref{prop:perspective}. On the other hand, we showed that while the perspective reformulation can be strengthened using polyhedral theory as discussed in \S\ref{sec:facets}, it is already quite strong. Our experiments on robust optimization with discrete uncertainty sets suggest that the perspective reformulation can be used as an accurate proxy, resulting in tractable approximations that outperform classical alternatives in the literature.


\section{Acknowledgments}

Andr\'es G\'omez was supported in part by grant 2006762 from the National Science Foundation, and by grant FA9550-22-1-0369 from the Air Force Office of Scientific Research.  Weijun Xie was supported in part by grants 2046426 and 2153607 from the National Science Foundation.

	\bibliographystyle{plainnat}
	\bibliography{main}

\begin{thebibliography}{25}
\providecommand{\natexlab}[1]{#1}
\providecommand{\url}[1]{\texttt{#1}}
\expandafter\ifx\csname urlstyle\endcsname\relax
  \providecommand{\doi}[1]{doi: #1}\else
  \providecommand{\doi}{doi: \begingroup \urlstyle{rm}\Url}\fi

\bibitem[Ahmed and Atamt{\"u}rk(2011)]{ahmed2011maximizing}
Shabbir Ahmed and Alper Atamt{\"u}rk.
\newblock Maximizing a class of submodular utility functions.
\newblock \emph{Mathematical Programming}, 128\penalty0 (1-2):\penalty0
  149--169, 2011.

\bibitem[Akt{\"u}rk et~al.(2009)Akt{\"u}rk, Atamt{\"u}rk, and
  G{\"u}rel]{akturk2009strong}
M~Selim Akt{\"u}rk, Alper Atamt{\"u}rk, and Sinan G{\"u}rel.
\newblock A strong conic quadratic reformulation for machine-job assignment
  with controllable processing times.
\newblock \emph{Operations Research Letters}, 37:\penalty0 187--191, 2009.

\bibitem[Atamt{\"u}rk and G{\'o}mez(2017)]{atamturk2017maximizing}
Alper Atamt{\"u}rk and Andr{\'e}s G{\'o}mez.
\newblock Maximizing a class of utility functions over the vertices of a
  polytope.
\newblock \emph{Operations Research}, 65:\penalty0 433--445, 2017.

\bibitem[Atamt\"urk and G{\'o}mez(2020)]{atamturk2020safe}
Alper Atamt\"urk and Andr{\'e}s G{\'o}mez.
\newblock Safe screening rules for {L0}-regression from perspective
  relaxations.
\newblock In \emph{International Conference on Machine Learning}, pages
  421--430. PMLR, 2020.

\bibitem[Atamt{\"u}rk and Narayanan(2008)]{atamturk2008b}
Alper Atamt{\"u}rk and Vishnu Narayanan.
\newblock Polymatroids and mean-risk minimization in discrete optimization.
\newblock \emph{Operations Research Letters}, 36\penalty0 (5):\penalty0
  618--622, 2008.

\bibitem[Bacci et~al.(2019)Bacci, Frangioni, Gentile, and
  Tavlaridis-Gyparakis]{bacci2019new}
Tiziano Bacci, Antonio Frangioni, Claudio Gentile, and Kostas
  Tavlaridis-Gyparakis.
\newblock New {MINLP} formulations for the unit commitment problems with
  ramping constraints.
\newblock \emph{Optimization}, 2019.

\bibitem[Ben-Tal and Nemirovski(2000)]{ben2000robust}
Aharon Ben-Tal and Arkadi Nemirovski.
\newblock Robust solutions of linear programming problems contaminated with
  uncertain data.
\newblock \emph{Mathematical Programming}, 88\penalty0 (3):\penalty0 411--424,
  2000.

\bibitem[Bertsimas and Sim(2004)]{bertsimas2004price}
Dimitris Bertsimas and Melvyn Sim.
\newblock The price of robustness.
\newblock \emph{Operations Research}, 52\penalty0 (1):\penalty0 35--53, 2004.

\bibitem[Borrero and Lozano(2021)]{borrero2021modeling}
Juan~S Borrero and Leonardo Lozano.
\newblock Modeling defender-attacker problems as robust linear programs with
  mixed-integer uncertainty sets.
\newblock \emph{INFORMS Journal on Computing}, 33\penalty0 (4):\penalty0
  1570--1589, 2021.

\bibitem[Ceria and Soares(1999)]{Ceria1999}
Sebasti{\'a}n Ceria and Jo{\~a}o Soares.
\newblock Convex programming for disjunctive convex optimization.
\newblock \emph{Mathematical Programming}, 86\penalty0 (3):\penalty0 595--614,
  1999.

\bibitem[d'Aspremont et~al.(2008)d'Aspremont, Bach, and
  El~Ghaoui]{d2008optimal}
Alexandre d'Aspremont, Francis Bach, and Laurent El~Ghaoui.
\newblock Optimal solutions for sparse principal component analysis.
\newblock \emph{Journal of Machine Learning Research}, 9\penalty0 (7), 2008.

\bibitem[Dey et~al.(2021)Dey, Mazumder, and Wang]{dey2021using}
Santanu~S Dey, Rahul Mazumder, and Guanyi Wang.
\newblock Using $\ell_1$-relaxation and integer programming to obtain dual
  bounds for sparse pca.
\newblock \emph{Operations Research}, 2021.

\bibitem[Frangioni and Gentile(2006)]{Frangioni2006}
Antonio Frangioni and Claudio Gentile.
\newblock Perspective cuts for a class of convex 0--1 mixed integer programs.
\newblock \emph{Mathematical Programming}, 106\penalty0 (2):\penalty0 225--236,
  2006.

\bibitem[G{\'o}mez(2021)]{gomez2021strong}
Andr{\'e}s G{\'o}mez.
\newblock Strong formulations for conic quadratic optimization with indicator
  variables.
\newblock \emph{Mathematical Programming}, 188\penalty0 (1):\penalty0 193--226,
  2021.

\bibitem[G{\"u}nl{\"u}k and Linderoth(2010)]{Gunluk2010}
Oktay G{\"u}nl{\"u}k and Jeff Linderoth.
\newblock Perspective reformulations of mixed integer nonlinear programs with
  indicator variables.
\newblock \emph{Mathematical Programming}, 124:\penalty0 183--205, 2010.

\bibitem[Kim et~al.(2021)Kim, Tawarmalani, and Richard]{kim2021convexification}
Jinhak Kim, Mohit Tawarmalani, and Jean-Philippe~P Richard.
\newblock Convexification of permutation-invariant sets and an application to
  sparse principal component analysis.
\newblock \emph{Mathematics of Operations Research}, 2021.

\bibitem[Li and Xie(2020)]{li2020exact}
Yongchun Li and Weijun Xie.
\newblock Exact and approximation algorithms for sparse {PCA}.
\newblock \emph{arXiv preprint arXiv:2008.12438}, 2020.

\bibitem[Nemhauser et~al.(1978)Nemhauser, Wolsey, and
  Fisher]{nemhauser1978analysis}
George~L Nemhauser, Laurence~A Wolsey, and Marshall~L Fisher.
\newblock An analysis of approximations for maximizing submodular set
  functions{—I}.
\newblock \emph{Mathematical Programming}, 14\penalty0 (1):\penalty0 265--294,
  1978.

\bibitem[Shi et~al.(2022)Shi, Prokopyev, and Zeng]{shi2022sequence}
Xueyu Shi, Oleg~A Prokopyev, and Bo~Zeng.
\newblock Sequence independent lifting for a set of submodular maximization
  problems.
\newblock \emph{Mathematical Programming}, pages 1--46, 2022.

\bibitem[Wei et~al.(2020)Wei, G{\'o}mez, and
  K{\"u}{\c{c}}{\"u}kyavuz]{wei2020convexification}
Linchuan Wei, Andr{\'e}s G{\'o}mez, and Simge K{\"u}{\c{c}}{\"u}kyavuz.
\newblock On the convexification of constrained quadratic optimization problems
  with indicator variables.
\newblock In \emph{International Conference on Integer Programming and
  Combinatorial Optimization}, pages 433--447. Springer, 2020.

\bibitem[Wei et~al.(2021)Wei, G\'omez, and K\"u\c{c}\"ukyavuz]{wei2021ideal}
Linchuan Wei, Andr\'es G\'omez, and Simge K\"u\c{c}\"ukyavuz.
\newblock Ideal formulations for constrained convex optimization problems with
  indicator variables.
\newblock \emph{Mathematical Programming}, 2021.

\bibitem[Wei et~al.(2022)Wei, Atamt{\"u}rk, G{\'o}mez, and
  K{\"u}{\c{c}}{\"u}kyavuz]{wei2022convex}
Linchuan Wei, Alper Atamt{\"u}rk, Andr{\'e}s G{\'o}mez, and Simge
  K{\"u}{\c{c}}{\"u}kyavuz.
\newblock On the convex hull of convex quadratic optimization problems with
  indicators.
\newblock \emph{arXiv preprint arXiv:2201.00387}, 2022.

\bibitem[Xie and Deng(2020)]{xie2020scalable}
Weijun Xie and Xinwei Deng.
\newblock Scalable algorithms for the sparse ridge regression.
\newblock \emph{SIAM Journal on Optimization}, 30\penalty0 (4):\penalty0
  3359--3386, 2020.

\bibitem[Yu and Ahmed(2017)]{yu2017maximizing}
Jiajin Yu and Shabbir Ahmed.
\newblock Maximizing a class of submodular utility functions with constraints.
\newblock \emph{Mathematical Programming}, 162\penalty0 (1-2):\penalty0
  145--164, 2017.

\bibitem[Zou et~al.(2006)Zou, Hastie, and Tibshirani]{zou2006sparse}
Hui Zou, Trevor Hastie, and Robert Tibshirani.
\newblock Sparse principal component analysis.
\newblock \emph{Journal of Computational and Graphical Statistics}, 15\penalty0
  (2):\penalty0 265--286, 2006.

\end{thebibliography}

\end{document}